\documentclass[12pt]{amsart}
\usepackage{amssymb}
\usepackage{epsfig}  		
\usepackage{epic,eepic}       
\usepackage{amsmath}
\usepackage{amsmath,amsfonts,amssymb}
\usepackage{graphicx}
\usepackage{amscd}
\usepackage{pb-diagram}
\usepackage{url}
\newtheorem{thm}{Theorem}[section]

\usepackage[usenames,dvipsnames]{xcolor}
\usepackage{color}

\theoremstyle{definition}

\theoremstyle{remark}

\numberwithin{equation}{section}

\usepackage[noend]{algorithmic}
\usepackage{algorithm}

\newcommand{\card}[1]{\ensuremath{\left\|#1\right\|}}
\newtheorem{lemma}{Lemma}[section]
\newcommand\pfun{\mathrel{\ooalign{\hfil$\mapstochar\mkern5mu$\hfil\cr$\to$\cr}}}

\begin{document}
\title[Persistent homology using iterated Morse decomposition]{Computing homology and persistent homology using iterated Morse decomposition}

\author{Pawe\l\ D\l otko}
\address{Institute of Computer Science, Jagiellonian University, Krakow, Poland}
\email{pawel.dlotko@ii.uj.edu.pl}

\author{Hubert Wagner}
\address{Institute of Computer Science, Jagiellonian University, Krakow, Poland}
\email{hubert.wagner@ii.uj.edu.pl}

\begin{abstract}
In this paper we present a new approach to computing homology (with field coefficients) and persistent homology. We use concepts from discrete Morse theory, to provide an algorithm which can be expressed solely in terms of simple graph theoretical operations. We use \emph{iterated Morse decomposition}, which allows us to sidetrack many problems related to the standard discrete Morse theory. In particular, this approach is provably correct in any dimension.
\end{abstract}

\maketitle
\section{Preview.}
\label{sec:preview}

In this section we outline the purpose of the paper, assuming reader's familiarity with some concepts from computational topology. All the concepts will be carefully explained later. In this paper we introduce a new method to compute homology and persistent homology over field coefficients. The method is based on discrete Morse theory and is designed to be graph-theoretic. 

We present a brief, intuitive illustration of our method. As an example, let us consider a triangulation of a Dunce hat presented in Figure~\ref{fig:dunceHat}.a. We want to remind that this space has trivial homology but nontrivial topology. We will use discrete Morse theory to simplify the space, while preserving the homology. First, let us build a discrete Morse matching on this triangulation. It is well known that a Dunce hat has no perfect\footnote{A Morse complex is \emph{perfect} if each critical cell corresponds to exactly one homology generator.} Morse complex~\cite{ayala}. Therefore, for any Morse matching we obtain some critical cells not corresponding to homology generators. The matching presented in Figure~\ref{fig:dunceHat}.b is optimal i.e. there are as few critical cells as possible.  Now the Morse boundary is computed using the V-paths marked in Figure~\ref{fig:dunceHat}.c. The resulting Morse complex (with $\mathbb{Z}_2$ coefficients) is shown in Figure~\ref{fig:dunceHat}.d. Normally, in order to compute homology of a chain complex, boundary matrix is produced and Smith Normal Form diagonalization is performed. However we would like to introduce an alternative approach. 

Using Kozlov's version of discrete Morse theory, we can iterate the Morse complex construction. In other words, we build a Morse complex of a Morse complex etc. So, here we compute a Morse matching of a chain complex presented in Figure~\ref{fig:dunceHat}.d. The only matched pair is marked with an arrow. Once the (iterated) Morse complex is computed, we are left with a single 0-dimensional cell. It cannot be paired anymore and corresponds to the only homological feature: the connected component. This way we have computed the homology of the Dunce hat.

\begin{figure}[!h]
\centerline{\includegraphics[scale=0.5]{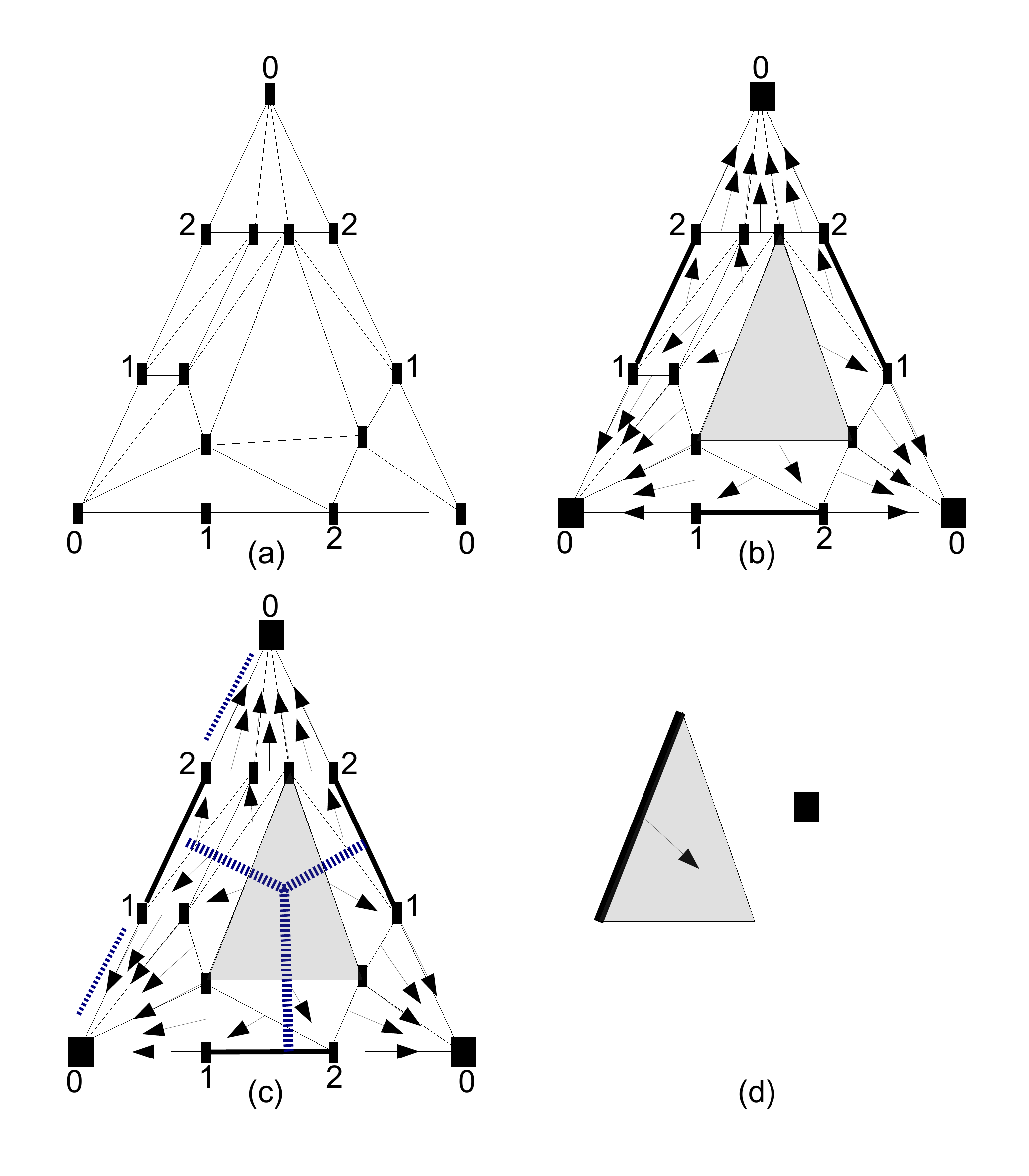}}
\caption
{
\emph{Iterated} version of Morse complex construction on a Dunce hat. On the top, the $0$ and $1$ dimensional critical cells are marked with bold, the middle gray triangle is the unique critical 2-cell. In the bottom left picture the V-paths used to compute the Morse complex after the first iteration of the construction. On the bottom right, the second (and final) iteration of Morse complex construction. The remaining vertex corresponds to the unique homology generator in dimension $0$.
}
\label{fig:dunceHat}
\end{figure}

Later the presented technique will be referred to as \emph{iterated Morse complex} construction or \emph{iterated Morse decomposition}. In this paper we will show that with the presented technique one can always acquire homology with field coefficients. We will also generalize it to the setting of persistent homology. As a result, we introduce a novel way to compute homology and persistence over a field. 
\section{Motivation}

While persistent homology can potentially be applied to a plethora of different practical problems, ranging from sensor networks~\cite{sensors} to root architecture analysis~\cite{herbert}, performance tends to be a problem. In particular, there is growing interest in analysis of high-dimensional topological features (going beyond connected components and 1-cycles). In low dimensions there exist efficient algorithms, but higher dimensional cases are still challenging.

Such applications include analysis of complex networks such as social-network and biological networks such as gene-regulatory networks. Computing homology or persistent homology of maps between spaces leads to high dimensional datasets as the studied space is the Cartesian product of the source and target space~\cite{harker}.

The only class of algorithms to compute persistent homology in the general case is based on matrix reductions. Such an approach was recently shown to expose roughly quadratic computational complexity, for data coming from certain practical applications~\cite{ctic}. Also, it is not very suitable for distributed computing, which is a necessity as the datasets grow larger.

Our aim is to propose an algorithmic framework which would scale reasonably well as the size of data (and its dimension) grows. In lower dimensions several algorithms were proposed and their efficiency stemmed from using fast techniques from graph-theory. Prompted by this observation, we wanted to propose a similar approach which would work for any dimension.

However, directly extending the existing approaches to higher dimensions is (as we believe) impossible. Combinatorial techniques, for instance described in~\cite{de, vanessa},  crucially depend on discrete Morse theory. In particular, a perfect Morse complex is (implicitly) constructed in most cases, which allows to read homology information right away. Computing such perfect Morse complexes is hard (or even impossible) in higher dimensions.

In case of field coefficients, which are important for practical applications, the situation is more tractable. We use the properties of \emph{iterated} Morse complexes \cite{ctic}, which always exist and are easy to compute. Additionally, we build on top of a recent theoretical framework by Mischaikow and Nanda \cite{mischaikow}, which extends Morse theory to filtrations of spaces.

The following paper describes the theory and algorithms for computing homology and persistent homology using \emph{iterated} Morse decomposition. We prove that the algorithm is correct for chain-complexes of any dimension. This includes commonly used simplicial and cubical complexes. Further, we show that, just as we intended, the algorithms can be entirely expressed in terms of basic graph-theoretical techniques. It promises that in terms of implementations, using efficient graph libraries~\cite{graph_lib} will result in a scalable solution.

The paper is structured as follows. Section~\ref{sec:previousWork} is a survey of existing work in the topic. In Section~\ref{sec:background} an introduction to homology and persistent homology is given, together with the necessary algorithmic background. In Section~\ref{sec:dmt} a Discrete Morse Theory is highlighted. In Section~\ref{sec:IteraterMorseComplex} the concept of \emph{iterated} Morse complex, crucial for this paper, is introduced. It is also explained there how this concept is used to compute homology. In Section~\ref{sec:iteratedMorseComplexForPersistence} these results are extended to simplification of filtered complexes, which is a preprocessing step for computing persistence. Later in Section~\ref{sec:persViaMorse} it is explained how to compute persistence using solely \emph{iterated} Morse complex. Finally in Section~\ref{sec:conclusions} conclusions are drawn.

\section{Previous work.}
\label{sec:previousWork}

Computations of homology and persistent homology is a well established area of research with a rich history. In this section we want to summarize the main contributions and historical landmarks in this subject. 

The classical way of computing homology is by using Smith Normal Form (SNF) of a boundary operator matrix, see~\cite{munkres}. The classical algorithm has hyper-cubical complexity (in case of integer coefficients), see~\cite{storjohann}. It is however possible to perform SNF in cubic time when field coefficients are considered. 

In the nineties Delfinado and Edelsbrunner provided an incremental algorithm for Betti numbers computation~\cite{de}. This algorithm exhibits linear time complexity and works for sub-triangulations of 3 dimensional sphere. There are also algorithms to compute homology with chain contractions and so called AT-models~\cite{rocio}.

In 2003, after a few earlier iterations (\cite{wlosi, wlosi1, vanessaPhd}), persistent homology was introduced in its contemporary form~\cite{firstPers}. In the same paper, a matrix-reduction algorithm to compute persistence (Algorithm~\ref{alg:reduction}) was given. This is what is considered the standard algorithm to compute persistence. For a comprehensive presentation the reader should consult~\cite{herbert}. Algorithmic results are further discussed in Section~\ref{sec:background}.

Discrete Morse theory (DMT) was introduced by Robin Forman~\cite{forman}. Later a more general, algebraic version was developed~\cite{kozlov}. The comprehensive presentation of algebraic discrete Morse theory can be found in~\cite{kozlov}. The idea of using DMT for homology computations has been introduced by Lewiner~\cite{lewinerHomology}. The complexity aspects of DMT has been discussed in~\cite{joswig}. The notions of $F-$perfect and $F-$optimal Morse complexes are discussed in~\cite{ayala}. A simplification algorithm for a Morse complexes on 2-manifolds has been presented in~\cite{wadzewski}. A divide and conquer algorithm to compute Morse complexes has been presented in~\cite{pascucci}.

Recently Robins, Wood and Sheppard have provided a practical link between discrete Morse theory and persistence~\cite{vanessa}. In this paper they introduce an optimal simplification scheme for persistence in case 3-dimensional complexes. The optimality of the presented result is restricted to 3 dimensions due to some deep results from simple homotopy theory. An extension of Morse theory suitable for simplifying filtered chain complexes in any dimension was later provided by Mischaikow and Nanda~\cite{mischaikow}. In short, by performing Morse matchings independently for each filtration level, a simplified filtered Morse complex is obtained. 

The idea of iterating Morse complex construction has already been used in~\cite{ctic, harker} as a tool to decrease the size of complexes before standard algebraic computations. 

There are many software libraries to compute homology and persistent homology. For a homology software the reader should consult~\cite{capd, chomp, kenzo}. For persistent homology~\cite{jplex, dionizous, perseus} are recommended.

\section{Background}
\label{sec:background}

\subsection{Complexes}
We assume that the input data is represented as a \emph{chain complex} with field coefficients. In the most typical case, this chain complex comes from a CW-decomposition of a given space which is a decomposition of a space into \emph{cells} of different dimensions. In practice, simplicial and cubical complexes are used. For simplicity, we will use $\mathbb{Z}_2$ coefficients throughout the paper, as this is the standard setting for persistence. However, we want to remark that the presented algorithms work for any field coefficients.

\subsection{Boundary maps}\quad
Let us fix a complex $\mathcal{K}$. Cells of $\mathcal{K}$ have different dimensions and are connected by boundary relations. If a $(p-1)$-cell $a$ has non-zero boundary coefficient with a $p$-cell $B$, we say $a$ is a proper face of $B$, and $B$ is a proper coface of $a$. (Notation: capital letters denote higher dimensional cell where a cell and its face is considered). 
Let a \emph{$p$-chain} be a formal sum of $p$-cells with the $\mathbb{Z}_2$ coefficients. The boundary operator $\partial_p$ maps $p$-chains into $p-1$-dimensional boundary chains. The chain of (co-)faces is called a (co-)boundary. We can extend the boundary operator linearly to $p$-chains. For any $p$-chain $c = \sum a_{i} c_{i}$, we have $\partial_{p}c = \sum a_i \partial_{p}c_{i}$. It is assumed that the boundary of a boundary is zero, or formally: $\partial_{p}\partial_{p+1} = 0$. The $p$-chains, together with addition modulo 2, form a \emph{group of $p$-chains}, denoted by $C_p$. 

The boundary operator $\partial_{p}$ can be written as a binary matrix (also denoted $\partial_p$), whose columns represent the boundaries and rows represent coboundaries of cells.



\subsection{Standard homology}
Intuitively, homology can be used to capture holes of complex $\mathcal{K}$. In 3-dimensional case holes are: connected components, tunnels, and voids. To define it formally, let us first introduce the group of $p$-cycles, $Z_p(\mathcal{K}) = ker \partial_p $ \and its subgroup: the group of $p$-boundaries, $B_p(\mathcal{K}) = im \partial_{p+1}$. The $p$-th homology group is the quotient $H_p = Z_p(\mathcal{K})/B_p(\mathcal{K})$. The $p$-th Betti number, denoted by $\beta_p$, is the rank of this group and counts the number of $p$-dimensional holes.

\subsection{Filtrations and persistence}
\label{sec:introToPersistence}
For a given complex $\mathcal{K}$, a filtration is defined as a nested sequence of its subcomplexes: $\emptyset = \mathcal{K}_{-1} \subseteq \mathcal{K}_{0} \subseteq \mathcal{K}_{1} \subseteq \ldots \subseteq \mathcal{K}_n=\mathcal{K}$~\cite{herbert}. In case of persistence, filtrations are often generated by a \emph{filtering function}, $g : \mathcal{K} \rightarrow \mathbb{Z}$ defined on the input complex. We require that $g(a) \leq g(B)$ whenever $a$ is a face of $B$. This property guarantees that the sub-level sets $\mathcal{K}_t=g^{-1}(-\infty,t]$ are sub-complexes of $\mathcal{K}$ for each value of $t \in \mathbb{Z}$. The inclusions from $\mathcal{K}_i$ to $\mathcal{K}_j$, for $i \leq j$ induce homomorphisms, $f^{i,j}: H(\mathcal{K}_i) \rightarrow  H(\mathcal{K}_j)$. Complex $\mathcal{K}$ with filtration will be refered to as \emph{filtered complex}. 

Given a complex $\mathcal{K}$ and a filtering function $g:\mathcal{K}\rightarrow \mathbb{Z}$, \emph{persistent} homology studies homological changes of the sub-level complexes, $\mathcal{K}_t=g^{-1}(-\infty,t]$.
Persistent homology captures the birth and death times of homology classes of the sub-level complexes, as $t$ grows from $-\infty$ to $+\infty$.
By birth, we mean that a homology feature is created; by death, we mean it either becomes trivial or becomes identical to some other class born earlier. 
The \emph{persistence}, or lifetime of a class, is the difference between the death and birth times. Often a multiset of \emph{persistence intervals} is used to represent persistence. An interval encodes a lifetime of a homology class of a given dimension. We say that two spaces have the same persistence, if their corresponding persistence intervals are the same.

The formal definition is as follows (after~\cite{herbert}): The $p$-th persistent homology groups of filtered complex $\mathcal{K}$ are the images of the homomorphisms induced by inclusion, $H^{i,j}(\mathcal{K}) = im f^{i,j}$. For a standard definition of persistence diagram and persistence intervals the reader should consult~\cite{herbert}.

We want to remind a theorem saying when persistence of two filtered complexes are equal:
\begin{thm}[Persistence equivalence theorem,~\cite{herbert}]
\label{thm:pet}
Consider persistent homology of two filtered complexes $X$ and $Y$. Let $\phi_i : H_{*}(X_i) \rightarrow H_{*}(Y_i)$:
$$
\begin{CD}
  H_{*}(X_0) @>>> H_{*}(X_1) @>>> \ldots @>>> H_{*}(X_{n-1}) @>>> H_{*}(X_{n}) \\
  @V\phi_0VV @V\phi_1VV @. @V\phi_{n-1}VV @V\phi_nVV\\
  H_{*}(Y_0) @>>> H_{*}(Y_1) @>>> \ldots @>>> H_{*}(Y_{n-1}) @>>> H_{*}(Y_n)\\
\end{CD}
$$

If the $\phi_i$ are isomorphisms and all the squares commute, then the persistence diagrams of $X$ and $Y$ are the same.
\end{thm} 




\subsection{Computing persistence}
Let us have a filtered chain complex $\mathcal{K}$.
Boundary matrix $\partial$ of $\mathcal{K}$ encodes the boundary relations between cells of different dimensions. Column $i$ corresponds to the boundary of cell $c_i$, row $j$ corresponds to the coboundary of cell $c_j$. In case of $\mathbb{Z}_2$ coefficients it can be defined as follows:

\[
	\partial({i,j}) = \begin{cases}1 & \text{ if } c_j \text{ is a face of } c_i\\ 0 & \text{ otherwise}\\ \end{cases}
\]

By $\kappa(c_i,c_j) := \partial({i,j})$ we denote the \emph{incidence index} of cells $c_i$ and $c_j$. In order to compute persistence, a \emph{sorted} boundary matrix is required: For two columns (or rows) $i < j$, the corresponding cells must satisfy: $g(c_i) \leq g(c_j)$. Using such a matrix, we can compute persistence using matrix-reduction Algorithm~\ref{alg:reduction} as defined in~\cite{herbert}. The value $low(i)$ marks the maximum (lowest) position of a one in column $i$, if any. We assume it to be zero for zeroed columns. We say that there is a \emph{collision} at column $j$, if there exists a column $k < j$ such that $low(k) = low(j)$, provided $low(k)$ and $low(j)$ are nonzero. 
The matrix is said to be reduced if there are no collisions. In such a case all the lowest ones are unique. As proven in~\cite{herbert}, persistent homology is fully determined by the positions of lowest ones in the reduced sorted matrix: If column $i$ is zero, corresponding p-dimensional cell $c_i$ creates a infinite p-dimensional persistent homology class. For a non-zero column $j$ with $k=low(j)$, the corresponding (p+1)-cell $c_j$ kills a persistent homology class created by p-cell $c_k$. 

The algorithm proceeds with columns from left to right, removing any collisions. Later we will analyze the behavior of the algorithm to prove the correctness of our simplification algorithm.

\begin{algorithm}[!ht]
  \small
  \caption{Compute reduced matrix} 
  \label{alg:reduction}
  \begin{algorithmic}[1]
  \REQUIRE Sorted binary matrix $\partial$ of size $n \times n$
  \ENSURE Reduced binary matrix $R$, which encodes persistence
    \STATE $R := \partial$ 
    \FOR{j := 1 to n}        
        \WHILE{there exists in $R$ a nonzero column $k < j$ with $low(k) = low(j)$}
            \STATE add column $k$ to column $j$ (mod 2) and store as column $j$
        \ENDWHILE
    \ENDFOR 
  \end{algorithmic}
  \label{alg:matRed}
\end{algorithm}

A simple illustration of Algorithm~\ref{alg:matRed} can be found in Figure~\ref{fig:persistenceAlgorithm} in the Appendix.

\subsection{Algorithms and their complexity}
Applying the presented matrix-reduction algorithm to the input complex is the standard way to compute persistent homology groups. It works for general complexes in arbitrary dimensions. The worst-case complexity is $O(n^3)$, where $n$ is the size of the input complex. Milosavljevic et al.~\cite{milosavljevic2011zigzag} showed that persistent homology can be computed in matrix multiplication time $O(n^\omega)$ where the currently best estimation of $\omega$ is $2.3727$. Chen and Kerber~\cite{chen2011output} proposed a randomized algorithm to compute only pairs with persistence above a chosen threshold. Despite improving the theoretical complexity, it is unclear whether these methods are better in practice.

When focusing on $0$-dimensional homology, union-find data structures can be used to compute persistence in time $O(n \alpha(n))$ \cite{herbert}, where $\alpha$ is the inverse of the Ackermann functions and $n$ the size of the input. 

A recent variation of the standard algorithm, introduced by Chen and Kerber~\cite{chao} significantly reduces the amount of computations. This idea was also used in~\cite{wagner11} to compute persistence for $n-$dimensional images. In general, the regular structure of cubical complexes can be exploited, which allows for handling large inputs. In such a situation the size of the boundary matrix is the main obstacle. Preprocessing the input complex using discrete Morse theory, as proposed by Robins~\cite{vanessa}, significantly reduces the size of the boundary matrix, while preserving persistnce. In case of 3D grayscale images, an efficient parallel implementation was proposed in~\cite{gunther}, allowing for handling large ($\approx 1200^3$) images on commodity hardware. The standard matrix-reduction algorithm is used in the final step of computations.

The approach by Robins works for arbitrary complexes and in dimension three the preprocessing results in the smallest possible boundary matrix~\cite{vanessa} (counting the number of rows/columns). The algorithm used in~\cite{vanessa} depends crucially on simple-homotopy theory, which makes it hard to directly generalize the optimality result to higher dimensions. A recent paper by Mischaikow et al.~\cite{mischaikow} proposes a handy theoretical framework, where discrete Morse theory is extended from complexes to filtrations. 

In our approach, we use the existing algorithms for discrete Morse complex construction. We use them to iteratively simplify the input complex. Note that our approach is significantly different from the simplification scheme by Pascucci et al., where Morse complexes are iteratively simplified in terms of (roughly speaking) topology or in a sense: persistent homology. Our aim is different: persistence in never affected, and the simplification is only in terms of the number of cells representing the complex.


\section{Discrete Morse Theory.}
\label{sec:dmt}
\subsection{Morse matching and Morse graph}

In this section an algorithmic introduction to discrete Morse theory is given. For further theoretical details please consult~\cite{forman,kozlov}. Let us have a complex $\mathcal{K}$. Discrete Morse theory partitions the cells of $\mathcal{K}$  into \emph{matched cells} and \emph{critical cells}. The critical cells, together with a boundary operator we describe later, form a chain complex called the \emph{Morse complex}. Importantly, this Morse complex has homology isomorphic with the homology of the initial complex. A procedure to compute Morse complex is given in Algorithm~\ref{alg:MorseCmplx}.

\begin{algorithm}[!ht]
  \small
  \caption{Compute Morse complex} 
  \label{alg:morseComplex}
  \begin{algorithmic}[1]
  \REQUIRE Input complex $\mathcal{K}$
  \ENSURE Resulting Morse complex
  \STATE $M$ := acyclic Morse matching on $\mathcal{K}$
  \STATE $C$ := list of critical cells of $M$
  \STATE $G$ := Morse graph of $M,C$
  \STATE $bd$ := compute Morse Boundary of $G$ (as in Algorithm~\ref{alg:morseBD})
  \STATE return ($C$, $bd$)
  \end{algorithmic}
  \label{alg:MorseCmplx}
\end{algorithm}

The first step in constructing a Morse complex requires finding an acyclic Morse matching $M$ of the complex $\mathcal{K}$. 
The matching is a partial map $M : \mathcal{K} \pfun   \mathcal{K}$. Each cell of $\mathcal{K}$ can be matched with exactly one of its co-faces. Some cells can remain unmatched, and are called \emph{critical}. These cells constitute the resulting Morse complex. 

Let us introduce a concept of a \emph{Morse graph} of a complex $\mathcal{K}$ and matching $M$. It is a directed graph whose vertices are formed by cells of a complex. A directed edge from vertex $A$ to vertex $b$ is added whenever $b$ is in the boundary of $A$ and the cells $A$ and $b$ are not matched in $M$. If they are matched in $M$, a direct edge from $b$ to $A$ is added in the Morse graph. The matching $M$ is called \emph{acyclic} if the corresponding Morse graph is a directed acyclic graph (DAG). The paths of this graph are often called $V-$paths. There are various strategies of obtaining acyclic Morse matchings. Later in this paper we assume that every matching is acyclic. 

A Morse matching is called \emph{perfect} if each unmatched cell corresponds to a homology generator of the original complex. (We choose to talk about (perfect) matchings, but in the literature (perfect) Morse functions, vector-fields and complexes are discussed.) This depends on the choice of coefficients, for example in case of a field coefficients $F$, we can talk about $F$-perfect matchings \cite{ayala}. Some spaces do not admit perfect matchings, for example the Dunce hat (presented in Section~\ref{sec:preview}), being contractible but non-collapsible. In this case some critical cells are spurious, in the sense that they do not correspond to any homology generators. 

One can try to construct a best possible matching, minimizing the number of critical cells. This problem is known to be NP-complete and MAXSNP-hard \cite{joswig}. It means that computing the best possible matching is computationally expensive and there is no hope to find a fully-polynomial time approximation strategy. As a result, no polynomial-time algorithm can give arbitrarily good bounds on the number of spurious critical cells.

Once an acyclic Morse matching $M$ is obtained, we proceed with computing the Morse boundary. This procedure is described in~\cite{forman}. The idea is illustrated in Figure~\ref{fig:MorseGraph}. Forman \cite{forman} proved that the resulting Morse complex has isomorphic homology to the homology of the initial complex. He also provided a formula which computes the boundary of each cell in a Morse complex. Kozlov generalized these proofs to the setting of arbitrary chain complexes \cite{kozlov}. One can use Morse complex construction to compute homology. Later we show that similar construction can be also used to compute persistence.

\subsection{Discrete Morse theory for filtered complexes.}
Recently there were first successful attempts to use discrete Morse theory to compute persistence~\cite{vanessa,gunther} (in case of 3-d gray-scale images), \cite{wadzewski} (in case of 2-manifolds). The first successful attempt to provide a Morse-theoretic categorical framework for persistent homology was made in~\cite{mischaikow}. 

In this section we will first recall the basic ideas from~\cite{mischaikow}. We say that the Morse matching $M$ is \emph{compatible with
 filtration} of $\mathcal{K}$ if for every matched $A \in \mathcal{K}$, $g(A) = g(M(A))$\footnote{By $M(A)$ we denote the element matched with $A$.}. In other words, the matchings are made between elements of the same filtration level. Consequently,  directed paths cannot move upwards the filtration (if they did, we would lose the sub-complex filtration property in the corresponding Morse complex, because a cell could enter the filtration strictly
 before its faces). 

The key result in~\cite{mischaikow} is that the persistence diagram of a filtered complex $\mathcal{K}$ and a Morse complex $\mathbb{M}(\mathcal{K})$ with the Morse matchings compatible with filtration are the same. In Figure~\ref{fig:compatibleField} an example of Morse matchings compatible and non-compatible with filtration are shown. 
\begin{figure}[!h]
\centerline{\includegraphics[scale=0.3]{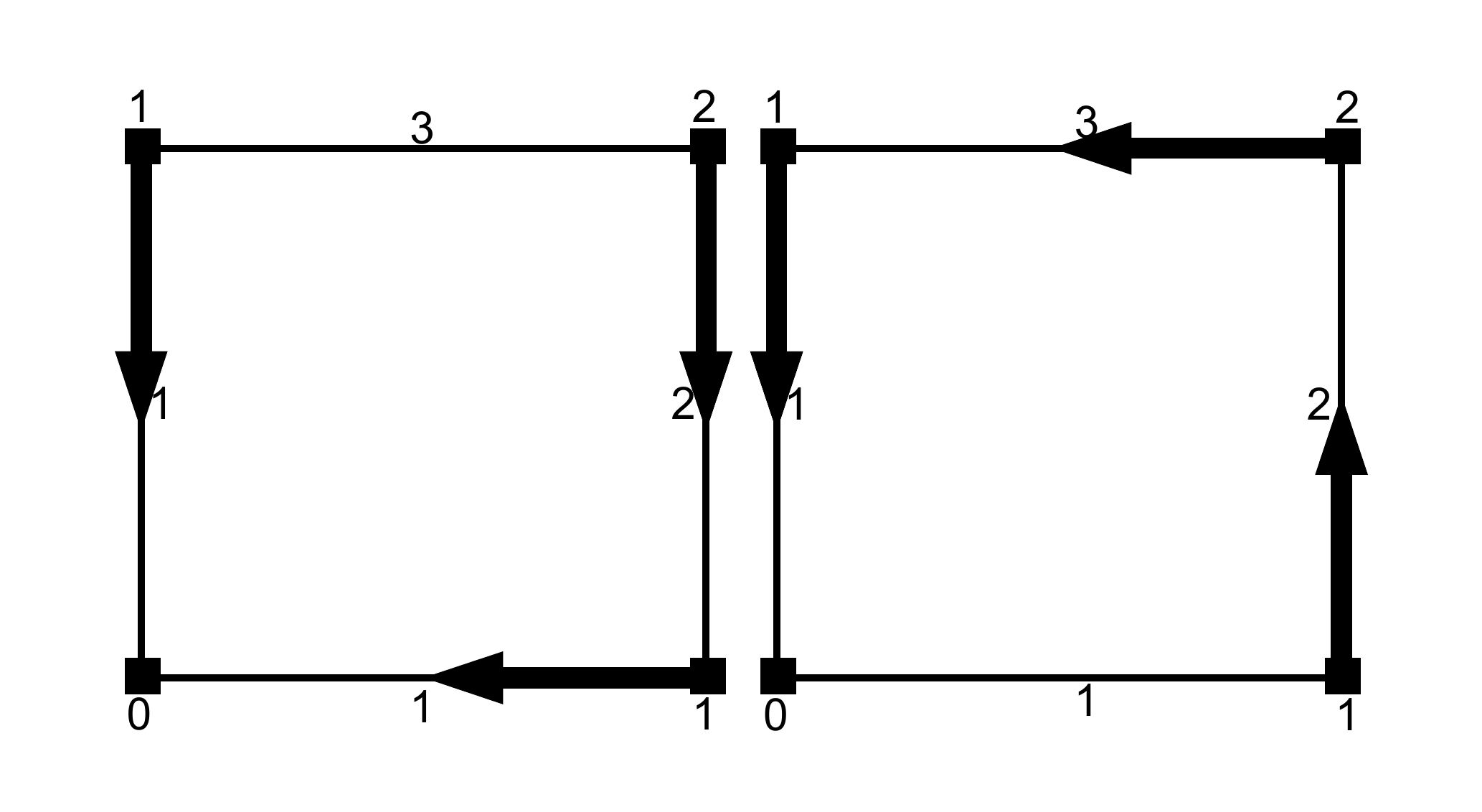}}
\caption{On the left the Morse matching compatible with filtration. In this case persistent homology for both initial and the Morse complex is $[0,\infty ]$ in dimension $0$ and $[3,\infty ]$ in dimension $1$. On the right a correct Morse matching in a sense of standard Discrete Morse Theory which is not compatible with filtration is depicted. The persistent homology of the Morse complex on the right in dimension one is $[1,\infty]$ which is not correct.}
\label{fig:compatibleField}
\end{figure}

\subsection{Computing Morse complex with graph algorithms}
In this section we will show that the entire Morse complex construction can be computed using standard graph algorithms. The chain complex (with $\mathbb{Z}_2$ coefficients) can be interpreted as a graph -- namely the Morse graph. We assume that initially no matchings are made. The whole construction can be divided into two essential parts: finding an acyclic Morse matching and computing the Morse boundary. Both parts are described below. We want to point out that for presentation's sake the algorithms presented in this section are not necessarily optimal. For that reason we also present just a version for binary coefficients. They can be easily generalized to arbitrary field coefficients. 

\begin{figure}[!h]
\centerline{\includegraphics[scale=0.5]{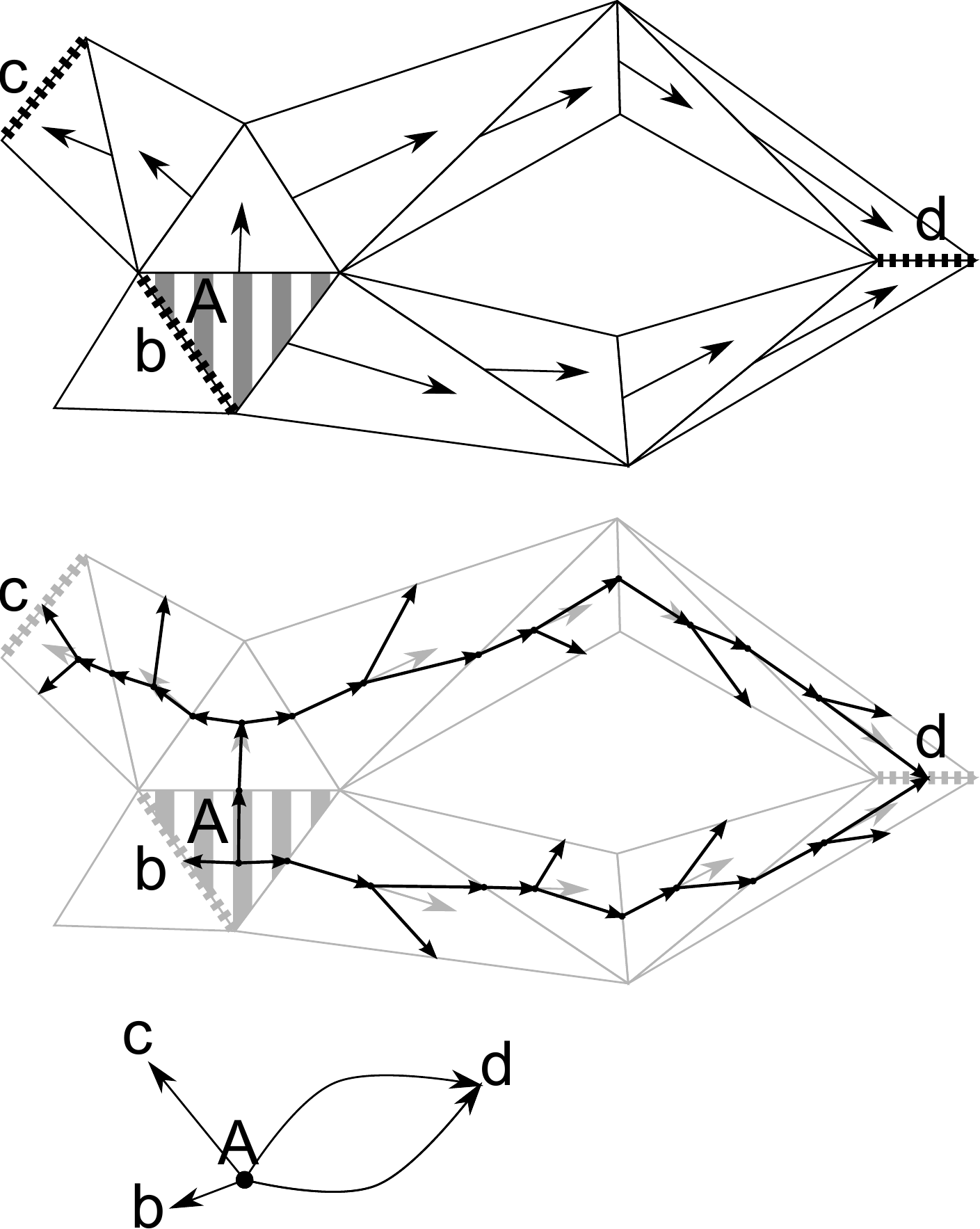}}
\caption
{
This picture shows the process of computing the Morse boundary of cell $A$. A Morse matching is shown. The part of corresponding (acyclic) Morse graph is build. Finally boundary relations between cells are computed, by following paths emanating from the boundary of cell $A$ and ending at critical cells. Note that there are two paths between cells $A$ and $d$, which yields boundary relation equal to zero for coefficients in $\mathbb{Z}_2$ ($d$ is not in the boundary of $A$).
}
\label{fig:MorseGraph}
\end{figure}

\subsubsection{Computing acyclic matching.}
There are many possible graph-based strategies for Morse matchings. Some strategies based on the idea of BFS spanning tree have been described in the literature~\cite{joswig,lewinerHomology,ctic}. In general, the problem of performing a Morse matching is equivalent to the following one -- which directed edges in Morse graph can be reversed so that the graph remains acyclic.  This is closely related to the \emph{minimum feedback arc set} problem. While this problem is NP-hard, efficient approximation schemes exist~\cite{even}. 

For an illustration, a basic algorithm descried in~\cite{cticKM} will be reminded in Algorithm~\ref{alg:morseViaCore}.
\begin{algorithm}[!ht]
  \small
  \caption{Compute Morse matching}   
  \begin{algorithmic}[1]
  \REQUIRE Morse graph $G$;
  \ENSURE Changed graph $G$ (edges between matched elements are reversed);
		\WHILE {Not all vertices of $G$ are marked}
			\STATE Let $i$ be the minimal dimension of unmarked element in $G$;
			\STATE Pick $A$, unmarked i-dimensional cell in $G$ and mark as \emph{critical};
			\WHILE {There exists unmatched element $A \in G$ with unique unmatched element $b$ in boundary} \label{restriction:here}				
				\STATE Make a Morse matching $(A,b)$, i.e. reverse an edge from $A$ to $b$ in $G$; 
				\STATE Mark $A$ and $b$ as matched;				
			\ENDWHILE
		\ENDWHILE
  \end{algorithmic}
  \label{alg:morseViaCore}
\end{algorithm}
The proof that the obtained graph is acyclic is easy but technical. Therefore it will not be presented here. 

To put restrictions on the matching one should modify line~\ref{restriction:here} of Algorithm~\ref{alg:morseViaCore}. In particular, one can ensure that the matchings are compatible with filtration. This is needed in case of persistence computations.

\subsubsection{Computing Morse boundaries.}

Let us state the problem of computing Morse boundary in terms of graph theory. The Morse boundary of a critical cell is formed by the set of critical cells which are reachable by an \emph{odd} number of paths in the Morse graph. This is a special case of Forman's formula~\cite{forman} in case of $\mathbb{Z}_2$ coefficients. Since the number of such paths can grow exponentially in the number of cells, brute force calculation is ineffective (as noted in \cite{mischaikow}). However, this problem can be solved efficiently, exploiting the fact that the Morse graph is acyclic. 

The following Algorithm~\ref{alg:morseBD} is simple to implement, and works in pessimistic time $O(c * (\|V\|+\|E\|))$, where $V$ and $E$ are respectively the vertices and edges of graph $G$ and $c$ is the number of critical cells. For a runtime- and memory-optimal one see~\cite{gunther}.

Let $P_{s}(t)$ denote the number of distinct paths leading from vertex $s$ to $t$, and $prev(v) = \{x \ | \ (x,v) \in E \}$ is the set of vertices preceding $v$ in the directed graph. We have an obvious recurrence relation:

\begin{displaymath}
P_s(u) =
\begin{cases}
 1 \text{ for } u = s \\
 \sum_{v \in prev(u)} {P_s(v)} \text{ for }  u \neq s
\end{cases}
\end{displaymath}

This recurrence can be computed directly and also efficiently using memoization, but we propose an elegant graph-theoretical algorithm. To compute $P_s(v)$ the summation is done indirectly in line \ref{neighbours_for2} of the Algorithm~\ref{alg:morseBoundaries} by adding the value $P_s(u)$ where $u \in prev(v)$.

\begin{algorithm}[!ht]
  \small
  \caption{Compute Morse boundaries from Morse graph} 
  \label{alg:morseBoundaries}
  \begin{algorithmic}[1]
  \REQUIRE Directed Morse graph $G := (V,E)$
  \ENSURE Boundary relation $\partial$ of the Morse complex
  \STATE sort G topologically
	\FOR{each critical vertex $s$}
		\STATE assign $P_s(v) := 0$ for each vertex $v \neq s$
		\STATE{assign $P_s(s) := 1$} \label{init}
	    \FOR{each vertex $c$ following $s$ in topological order} \label{inner_for}	
			\IF{$c$ is critical}
			\STATE{$\partial(s,c)$ := $P_s(c)$ mod 2}									
			\ELSE
			\FOR{each $v \ | \ (c,v) \in E$ } \label{neighbours_for}
				\STATE{$P_s(v)$ += $P_s(c)$} \label{neighbours_for2}
			\ENDFOR   
			\ENDIF			
		\ENDFOR   
	\ENDFOR
  \end{algorithmic}
  \label{alg:morseBD}
\end{algorithm}

\begin{thm}
Algorithm~\ref{alg:morseBoundaries} is correct.
\end{thm}

\begin{proof}
We prove the correctness of the algorithm by induction on the iteration of the loop in line~\ref{inner_for}. The desired invariant is that whenever $P_s(c)$ is used, this value is already final and correct. Clearly, the value $P_s(s)$ is initialized correctly in line~\ref{init}, so the correct value is used during the first iteration. Assume that the invariant holds for the first $i$ iterations and vertex $c$ is now processed. Note that the value of $c$ depends on the values of $prev(c)$. Since we proceed in topological-sort order, all the vertices in $prev(c)$ have already been processed. By inductive assumption the values used to compute $P_s(c)$ were correct and final, therefore the value $P_{s}(c)$ is also correct and final.
\end{proof}

\section{Iterated Morse Complex for homology}

\label{sec:IteraterMorseComplex}
In this section the concept of \emph{iterated Morse complex} is presented. Normally one aims at finding a Morse complex minimizing the number of critical cells. As mentioned earlier this is a hard algorithmic problem and we do not tackle it. Instead we use an algorithm to iteratively construct a sequence of Morse complexes. If at a certain stage the obtained Morse complex is far from optimal, further iterations will be necessary to compute the homology of the considered complex. Still, the worst case computational time is cubical. The results presented in this section has already been sketched in~\cite{ctic}. 

Let $\mathcal{C}$ be a category of chain complexes and let $\mathbb{M} : \mathcal{C} \rightarrow \mathcal{C}$ be a functor taking a chain complex and assigning it a Morse complex constructed on it. There are many possible strategies to construct Morse complexes. We assume that if a chain complex $\mathcal{K} \in \mathcal{C}$ has some available Morse matchings, $\mathbb{M}$ does at least one of them\footnote{The simplest example of algorithm that fulfill this requirement is $\mathbb{M}$ which searches for a first possible Morse matching in $\mathcal{K}$, makes it, and computes the Morse complex.}. This property of $\mathbb{M}$ will be referred to as \emph{vitality}. Except from vitality no extra assumptions are put on $\mathbb{M}$. 

For a given chain complex $\mathcal{K} \in \mathcal{C}$, \emph{iterated} Morse complex $\mathbb{M}^{\infty}(\mathcal{K})$ is the fixed point of the iteration $\mathbb{M}(\mathcal{K}), \mathbb{M}^2(\mathcal{K})
 = \mathbb{M}(\mathbb{M}(\mathcal{K})), \mathbb{M}^3(\mathcal{K}), \ldots$. It is clear that $\card{\mathcal{K}} \geq \card{\mathbb{M}(\mathcal{K})} \geq
 \card{\mathbb{M}^2(\mathcal{K})} \geq \ldots$. Moreover due to the vitality of $\mathbb{M}$ the above inequalities are strict as long as there are some
 Morse matchings to be made in the intermediate complexes. Therefore, the fixed point $\mathbb{M}^{\infty}(\mathcal{K})$ is obtained in a finite number of iterations. Below we
 show that $\mathbb{M}^{\infty}(\mathcal{K})$ gives an instant information about homology of $\mathcal{K}$. To achieve this, we will use algebraic version of discrete Morse theory
 due to Kozlov~\cite{kozlov}. It states that two elements $A,B$ can be matched if and only if $\kappa(A,B)$ is invertible. Since in this paper we
 consider only homology with field coefficients, $\kappa(A,B)$ is always invertible provided it is nonzero. This fact implies the following straightforward
 lemmas:
\begin{lemma}
For every $A \in \mathbb{M}^{\infty}(\mathcal{K})$ both boundary and coboundary of $A$ are empty.
\end{lemma}
\begin{lemma}
$\beta_i(\mathcal{K}) = \card{\{ A \in \mathbb{M}^{\infty}(\mathcal{K}) | \dim A = i \}}$.
\end{lemma}

The proof of the first lemma is a direct consequence of vitality of $\mathbb{M}$. If there exists $A \in \mathbb{M}^{\infty}(\mathcal{K})$ with $B$ in (co)boundary, then $\mathbb{M}$ would eventually make a Morse matching between $A$ and $B$. 

The second lemma is a direct consequence of the first one. Once every element has empty boundary, it is a cycle. Once it has empty coboundary, it cannot be a boundary. Therefore every element in $\mathbb{M}^{\infty}(\mathcal{K})$ generates a homology class of $\mathcal{K}$. Moreover, due to Section 11.3 in~\cite{kozlov} it is clear that the homology of $\mathbb{M}^i(\mathcal{K})$ and $\mathbb{M}^{i+1}(\mathcal{K})$ are isomorphic. Therefore the homologies of the complex are preserved through the entire iteration. 

As already mentioned, the idea of iteration of Morse complex construction implies that we do not have to construct near optimal Morse complexes. Let $n$ be the cardinality of $\mathcal{K}$. In the worst case, after $\lfloor n/2 \rfloor$ iterations of Morse complex procedure, an \emph{iterated} Morse complex is obtained\footnote{Maximal number of iterations needed is the floor of cardinality of basis of $\mathcal{K}$ divided by two minus sum of all the Betti numbers of $\mathcal{K}$.}. This is a consequence of vitality of $\mathbb{M}$. 

\begin{algorithm}[!ht]
  \small
  \caption{Compute homology with iterated Morse decomposition} 
  \label{alg:homology}
  \begin{algorithmic}[1]
  \REQUIRE Initial complex $C$ of dimension $d$
  \ENSURE Betti numbers $\beta_i$    
    \WHILE{true}        
            \STATE M := Build Morse complex of C (Algorithm~\ref{alg:morseComplex})
            \IF{M = C}
				\STATE break
            \ENDIF 
			C := M	
    \ENDWHILE
    
    \FOR{i := 0 to d}        
	\STATE $\beta_i$ := number of i-dimensional cells of $C$
	\ENDFOR
    
  \end{algorithmic}  
\end{algorithm}

Algorithm~\ref{alg:homology} describes how to comute the Betti numbers using iterated Morse decomposition. An example of the Morse complex construction on a Dunce hat has already been presented in Section~\ref{sec:preview}. In case of the Dunce hat there does not exist a perfect Morse complex. At the end of this section, in Figure~\ref{fig:iteratedCmplxPers}, we show a simple example of the presented construction, using a sub-optimal algorithm $\mathbb{M}$ to construct Morse complexes with sub-optimal Morse matchings. 

\section{Iterated Morse Complex for persistent homology}
\label{sec:iteratedMorseComplexForPersistence}

In~\cite{mischaikow} it is shown that when a Morse complex is constructed based on a Morse matching compatible with filtration, persistent homology of the initial complex and the one of the Morse complex are isomorphic. Here we provide a further consequence of this result. Let us take a vital functor $\mathbb{M}$ acting from a category of \emph{filtered} chain complexes to itself. We assume that the Morse matching used to construct $\mathbb{M}(\mathcal{K})$ is compatible with filtration of $\mathcal{K}$. Filtration values of cells in $\mathbb{M}(\mathcal{K})$ are inherited from the filtration of cells in $\mathcal{K}$. As in Section~\ref{sec:IteraterMorseComplex}, we construct $\mathbb{M}^{\infty}(\mathcal{K})$.

In this section we show that each cell in $\mathbb{M}^{\infty}(\mathcal{K})$ either creates or kills a feature of nonzero persistence. Therefore, if we want to minimize the number of cells, the resulting complex is the minimal complex encoding persistence of the original complex. An example is presented in Figure~\ref{fig:morseIterExample}.
\begin{figure}[!h]
\centerline{\includegraphics[scale=0.5]{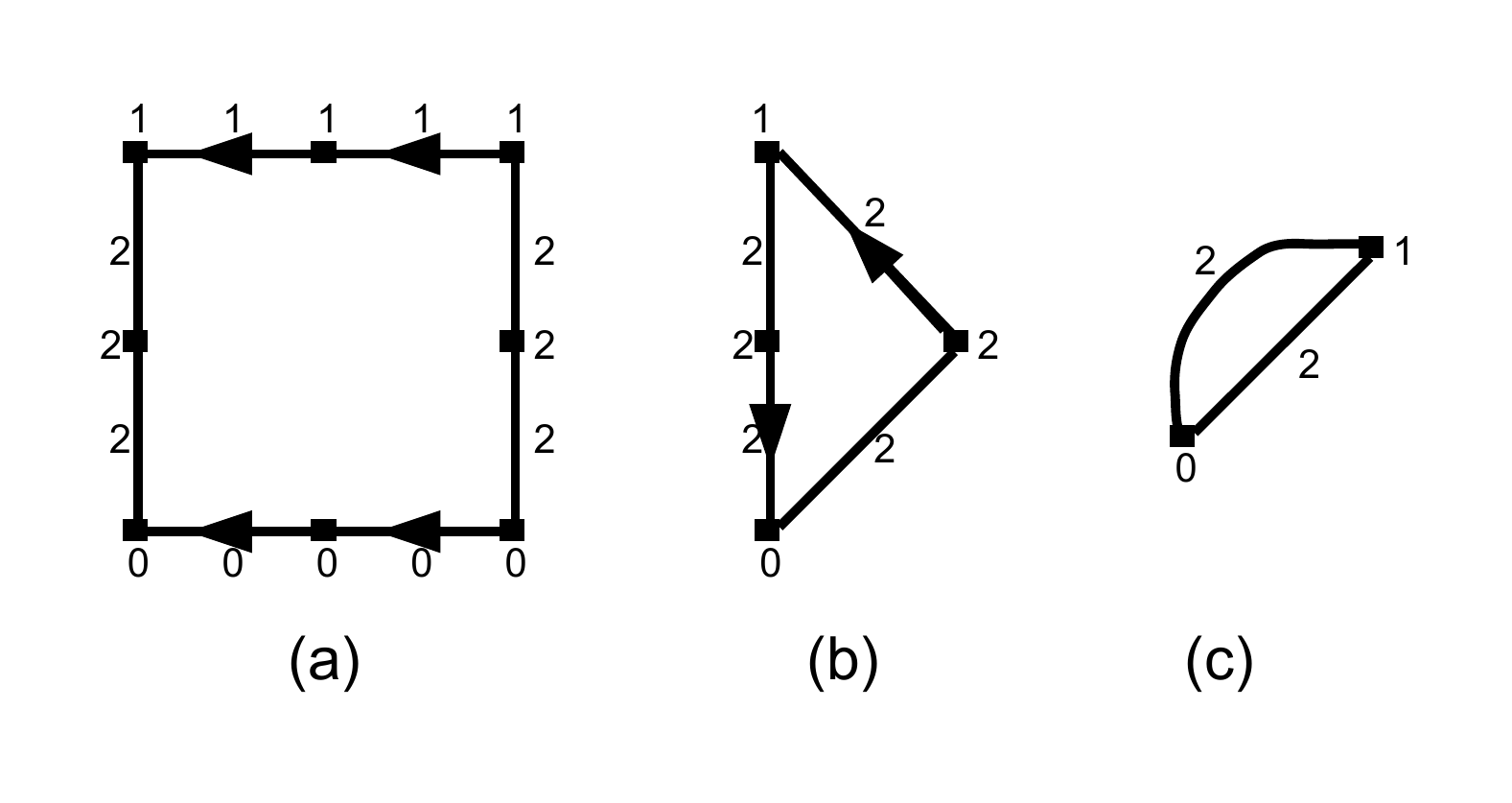}}
\caption{(a) The initial complex $\mathcal{K}$ with a few Morse matchings indicated with arrows. (b) Morse complex $\mathbb{M}^1(\mathcal{K})$ obtained from $\mathcal{K}$. A few possible Morse matchings indicated with arrows. (c) Final Morse complex $\mathbb{M}^2(\mathcal{K})$  obtained from $\mathbb{M}^1(\mathcal{K})$.}
\label{fig:morseIterExample}
\end{figure}

In~\cite{mischaikow} it is only assumed that the filtered chain complex is given at the input. Since at each stage of \emph{iterated} Morse complex computation we have a  chain complex, one can iterate further the construction. The main strength of our approach is based on the following novel observation. The resulting complex $\mathbb{M}^{\infty}(\mathcal{K})$ has the following property: For every $A \in \mathbb{M}^{\infty}(\mathcal{K})$ and for every $b_1,\ldots,b_n$ in boundary of $A$ we have $g(A) > g(b_1),\ldots,g(b_n)$. It is because if there existed $b_i$ in the boundary of $A$ such that $g(b_i) = g(A)$, then a Morse matching could be made between $A$ and $b_i$ (since coefficients in a field are used). This would contradict the vitality assumption of $\mathbb{M}$. Having this simple property of the resulting complex $\mathbb{M}^{\infty}(\mathcal{K})$ we can now present the main theorem of this section:
\begin{thm}
\label{thm:optimalMorse}
Let $\mathbb{M}^{\infty}(\mathcal{K})$ be the iterated Morse complex obtained from the initial filtered chain complex $\mathcal{K}$ by iterative construction of Morse complexes using Morse matchings compatible with filtration. Then:\\
$\mathcal{K}$ and $\mathbb{M}^{\infty}(\mathcal{K})$ have the same persistence. \emph{(Correctness)} 
\\
Every element $A \in \mathbb{M}^{\infty}(\mathcal{K})$ either starts or terminates a nonzero length persistence interval. \emph{(Optimality)} 
\end{thm}
\begin{proof}

\emph{Correctness}:
To show that the algorithm is correct we will apply the Persistence Equivalence Theorem for each iteration:

\[
\begin{CD}
  \ldots @>j_{l-1}>> H_{*}(\mathbb{M}^{i}_{l}) @>j_{l}>> H_{*}(\mathbb{M}^{i}_{l+1}) @>j_{l+1}>> \ldots \\
  @VVV@VV  bd^{\mathbb{M}}_{l} V@VV bd^{\mathbb{M}}_{l+1} V @VVV\\
  \ldots @>k_{l-1}>> H_{*}(\mathbb{M}^{i+1}_{l}) @>k_{l}>> H_{*}(\mathbb{M}^{i+1}_{l+1}) @>k_{l+1}>> \ldots\\
\end{CD}
\]

We need to show two things:
\begin{enumerate}
\item That the vertical maps are isomorphisms on homology level. 
\item That all squares commute.
\end{enumerate}

First note that the vertical maps send each chain in the input complex to a corresponding chain of the Morse complex. This is equivalent to computing Morse boundaries, as described in Section~\ref{sec:dmt}. We remind that to find a corresponding Morse chain we follow appropriate V-paths in the Morse graph. 

1) Vertical arrows are isomorphisms: this is a consequence of Theorem 2.1 from \cite{kozlov}, which states that the upper chain complex $M$ is decomposed by the Morse construction into an \emph{acyclic} part and the Morse complex having homology isomorphic with $M$.

2) To prove that the squares commute for each $i, l$, let us take a chain $c \in \mathbb{M}^i_l = \sum c_j$. We show that $(bd^{\mathbb{M}}_{l+1} \circ  j_l) (c) = (k_l \circ bd^{\mathbb{M}}_{l}) (c)$.

Down and right ($k_l \circ bd^{\mathbb{M}}_{l}$): If $c_j$ is critical, it is unchanged by the vertical map. Otherwise, we follow the paths of the Morse graph to compute the corresponding chain in the Morse complex. Repeating this computation for every $c_j$, the value of $bd^{\mathbb{M}}_{l}$ on $c$ is obtained.  Moving right with inclusion, the chain remains the same.

Right and down ($bd^{\mathbb{M}}_{l+1} \circ j_l$): first the chain $c$ is inserted by inclusion into level $l+1$ of filtration, so it is unchanged. But now we move with the vertical arrow, which might be richer on this level, as additional paths enter the Morse graph. Note that since we force the paths to be non-increasing with filtration, $bd^{\mathbb{M}}_{l+1}$ restricted to level $l$ is the same as $bd^{\mathbb{M}}_{l}$. In other words, any V-path starting at $c_j$ at level at most $l$ can only reach cells of lower or equal filtration values. In particular, it will never reach any critical cell introduced at level $l+1$. 

Therefore the two images of chain $c$ are the same and the diagram commutes.

We can apply Persistence Equivalence Theorem and finish the proof that persistent homology is unchanged during our \emph{iterated} Morse complex construction.

\emph{Optimality}:

We want to show that the simplified complex, $\mathbb{M}^{\infty}(\mathcal{K})$ contains only significant information, that is \emph{no intervals of persistence zero} are present. The argument is based on the analysis of the behavior of the standard (left-to-right) matrix-reduction algorithm (Algorithm~\ref{alg:reduction}) run on the boundary matrix of the final \emph{iterated} Morse complex $\mathbb{M}^{\infty}(\mathcal{K})$. The lowest-ones of the reduced matrix directly indicate persistence intervals. Zero columns indicate that a given cell creates an infinite interval.

The argument is inductive with respect to the iteration of the outer loop in the Algorithm~\ref{alg:matRed}. Specifically, we consider the first $k$ reduced columns of the matrix. For $k=0$ this submatrix is empty, so there are no zero-persistence pairs.

Let us assume that the argument holds for some $k \geq 0$. 

Suppose by contradiction that there is a zero-length persistence interval generated by a pair $(A,b)$, where $A$ is the $k+1$ column. It means that at this stage we have the following situation (dots mark arbitrary entries):

The matrix after $k+1$ iterations (first $k+1$ columns are reduced).
\[
\begin{array}{ccccc}
&      & &k+1& \\
 &     & &A& \\
 &     &.&.&.\\
 &     & &.& \\
b&0 ...&0&1& \\
 &     & &0& \\
 &     &.&0&.\\
 &     & &0& \\
\end{array}
\]

The matrix after $k$ iterations (first $k$ columns are reduced).

\[
\begin{array}{ccccc}
 &       & &k+1& \\
 &    A1 & &A0& \\
 &    . &.&.&.\\
 &    . & &.& \\
b&.    .&...&.& \\
 &    . & &0& \\
 &    . &.&.&.\\
b1&   1 & &1& \\
\end{array}
\]

There are two possibilities:
\begin{enumerate}
\item During the reduction of $A0$ there were no collisions, so $A0=A$ is a cell in the complex $\mathbb{M}^{\infty}(\mathcal{K})$. Then, since $g(A) = g(b)$ and $\kappa(A,b) \neq 0$,  it was possible to make a Morse matching between $A$ and $b$, which gives a contradiction with the fact that all possible Morse matchings compatible with filtration were made in $\mathbb{M}^{\infty}(\mathcal{K})$.
\item $A$ is represented as a sum of the preceding columns, which generated collisions. In this case $A = A0 + A1+ \ldots An$, for $A0$ being the column in the unreduced matrix and $A1,\ldots,An$ being columns preceding $A0$ in the matrix. 

For the proof we need to find the lowest nonzero position of all the columns $A_i$, $i \in \{1,\ldots,n\}$. During the process of reducing a single column the lowest-ones can never increase, therefore the first collision yields the lowest-one we search for. We call the column $A1$ and this lowest position $b1$ and note that $g(b) \leq g(b1)$. Also note that $b1$ marks the lowest one in the reduced column $A1$ and the original column $A0$, so $g(b1) \leq g(A1)$ and $g(b1) \leq g(A0) = g(A)$.

From the assumption we have $g(A) = g(b)$. From the filtration of the complex we have $g(A1) \leq g(A)$ and $g(b) \leq g(b1)$. Putting this together we get:
$g(A) = g(b) \leq g(b1) \leq g(A1) \leq g(A)$, consequently $g(b1) = g(A1)$. This means that there was a zero-persistence pair within the first $k$ columns. This contradicts the inductive assumption.
\end{enumerate}
\end{proof}

We have the following theorem which is a direct consequence of Theorem~\ref{thm:optimalMorse}:
\begin{thm}
\label{thm:lastPhaseIsShort}
Let $\mathcal{K}$ be the initial filtered chain complex. Let $p$ be the number of finite and $k$ the number of infinite persistence intervals of $\mathcal{K}$. Then $\card{\mathbb{M}^{\infty}(\mathcal{K})} = 2p+k$.
\end{thm}

This theorem indicates that the complex $\mathbb{M}^{\infty}(\mathcal{K})$ is the minimal complex encoding the persistence of the initial complex. We can now apply the matrix reduction method, to get persistence intervals in time $O((2p+k)^3)$. Therefore if the number of persistence intervals is small, this computation can be efficient. As an alternative, we propose a new algorithm presented in Section~\ref{sec:persViaMorse}, which relies only on graph operations. 

Algorithm~\ref{alg:simplification} describes the simplification procedure. To compute persistence based on simplified complex $C$ use Algorithm~\ref{alg:reduction}. We want to point out that there is no obvious way of relaxing the condition $g(A) = g(M(A))$ for matched elements. See the Appendix for more details.

\begin{algorithm}[!ht]
  \small
  \caption{Simplification for persistence computations} 
  \label{alg:simplification}
  \begin{algorithmic}[1]
  \REQUIRE Initial filtered complex $C$ of dimension $d$
  \ENSURE Simplified complex $C$, having the same persistent homology
    \WHILE{true}        
            \STATE M := Build Morse complex of C using only matchings compatible with filtration
            \IF{M = C}
				\STATE break
            \ENDIF 
			C := M	
    \ENDWHILE
  \end{algorithmic}     
\end{algorithm}

At the end of this section, in Figure~\ref{fig:iteratedCmplxPers}  we present an example of the \emph{iterated} Morse complex construction. 
\begin{figure}[!h]
\centerline{\includegraphics[scale=0.5]{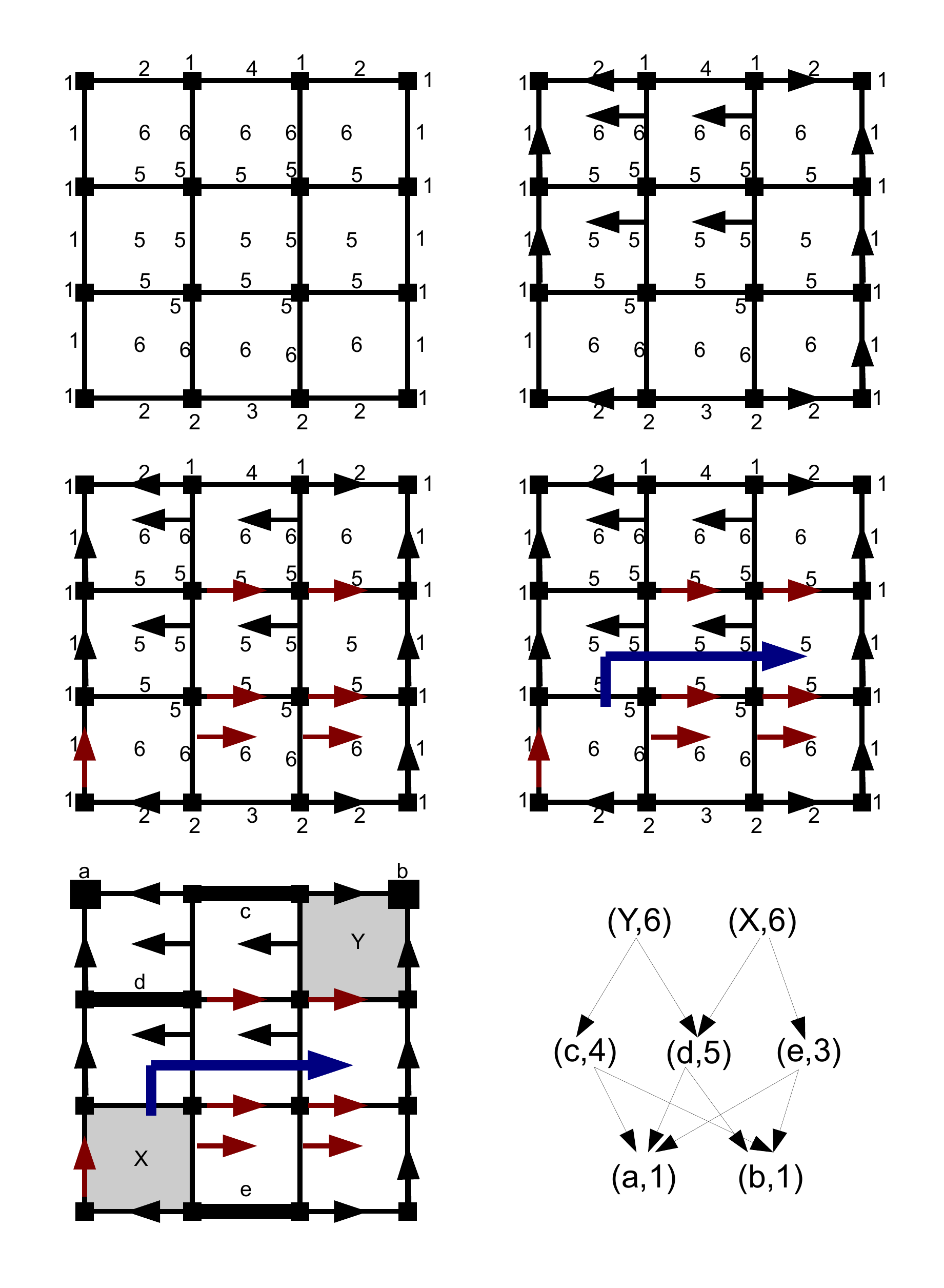}}
\caption{On the top left the initial filtered complex $\mathcal{K}$ is depicted. On the top right, the first iteration, $\mathbb{M}^{1}(\mathcal{K})$ of the Morse complex compatible with filtration construction. On the middle left with red arrows the second iteration $\mathbb{M}^{2}(\mathcal{K})$, and on the middle right the third and final iteration of $\mathbb{M}^{\infty}(\mathcal{K})$ construction. On the bottom left, the cells of $\mathbb{M}^{\infty}(\mathcal{K})$ are named, and on bottom right the boundary relation is presented in a form of diagram. The levels indicate the gradation -- vertices $a$, $b$ at the bottom, edges $c,d,e$ in the middle and faces $X$ and $Y$ at the top. The numbers indicate filtration values of elements and arrows -- boundary relation.}
\label{fig:iteratedCmplxPers}
\end{figure}

\section{Persistence intervals via iterated Morse approach.}
\label{sec:persViaMorse}

In this section we compute persistence intervals using the \emph{iterated} Morse complex approach. It will be shown that the presented approach can be interpreted as a variation of matrix-reduction algorithm. It is however based on graph theory, rather than matrix algebra. 

We stress that in this section, unlike previous one, we allow to make Morse matchings between elements having different level of filtration. Therefore, we will construct Morse matchings that are \emph{not} compatible with filtration in a sense that element $A$ can be matched with $b$ if $g(A) \geq g(b)$.

Algorithm~\ref{alg:persViaMorse} is used to compute persistence intervals of $\mathbb{M}^{\infty}(\mathcal{K})$.
\begin{algorithm}[!ht]
  \small
  \caption{Compute persistent homology via Iterated Morse complex.} 
  \label{alg:persViaMorse}
  \begin{algorithmic}[1]
  \REQUIRE Filtered iterated Morse complex $C = \mathbb{M}^{\infty}(\mathcal{K})$.
  \ENSURE Persistence intervals of $\mathbb{M}^{\infty}(\mathcal{K})$.
  \STATE $S :=$ empty multiset of persistence intervals;
  \STATE int $f :=$ second filtration level of $C$;
  \STATE int $l :=$ last filtration level of $C$;
  \STATE $M$ is a vital strategy to make Morse matchings. Matchings between elements of different filtrations are allowed in $M$. If element $A$ can be matched with two (or more) elements $b_1$, $b_2...$, we match it with the one with maximal filtration value.
  \FOR { int $i := f$ to $l$}
		\WHILE {true}
			\STATE Construct Morse matching on $C_{i}$ using $M$ (remark: for every $(A,b)$ matched by $M$ we have $g(b) < g(A) = i$.). 
			\IF {Nothing was matched by $M$}
				\STATE break
			\ENDIF
			\FOR {Every $(A,b)$ matched by $M$} 
					\STATE $S := S \cup [g(b),g(A)]$;
			\ENDFOR
			\STATE $C := $ Morse complex on $C$ constructed based on Morse matching $M$;
		\ENDWHILE
  \ENDFOR
  \FOR { every cell $c$ in complex $C$}
		\STATE $S := S \cup [g(c),\infty]$;
  \ENDFOR
  \RETURN $S$;
  \end{algorithmic}
  \label{alg:persistenceViaMorse}
\end{algorithm}

We want to point out that the Morse boundary procedure is always performed on the whole complex $\mathbb{M}^{\infty}(\mathcal{K})$ even if the matchings are made on a proper subcomplex $\mathbb{M}^{\infty}_i(\mathcal{K})$. We start from the complex $\mathbb{M}^{\infty}_{f}(\mathcal{K})$. Since $\mathbb{M}^{\infty}(\mathcal{K})$ is an iterated Morse complex, is clear that in $\mathbb{M}^{\infty}_{f-1}(\mathcal{K})$ and $\mathbb{M}^{\infty}_{f}(\mathcal{K}) \setminus \mathbb{M}^{\infty}_{f-1}(\mathcal{K})$ there are no Morse matchings to be made. But in $\mathbb{M}^{\infty}_{f}(\mathcal{K})$ there can be a Morse matching $(A,b)$ such that $A \in \mathbb{M}^{\infty}_{f}(\mathcal{K}) \setminus \mathbb{M}^{\infty}_{f-1}(\mathcal{K})$ and $b \in \mathbb{M}^{\infty}_{f-1}(\mathcal{K})$. This means that $b \in \mathbb{M}^{\infty}_{f-1}(\mathcal{K})$ is a homology generator in $\mathbb{M}^{\infty}_{f-1}(\mathcal{K})$  which is killed in $\mathbb{M}^{\infty}_{f}(\mathcal{K})$ (since the matching $(A,b)$ can be made without changing homology of $\mathbb{M}^{\infty}_{f}(\mathcal{K})$). Making such a matching indicates a persistence interval generated by the pair $(A,b)$, since the homology class generated by $b$ is killed by $A$. We assume that all possible matchings in $\mathbb{M}^{\infty}_{f}(\mathcal{K})$ are made (by iterating Morse complex construction) before proceeding to $\mathbb{M}^{\infty}_{f+1}(\mathcal{K})$.

When processing complex $\mathbb{M}^{\infty}_{i}(\mathcal{K})$ we assume that in $\mathbb{M}^{\infty}_{i-1}(\mathcal{K})$ no more Morse matchings can be made. We are searching for matchings $(A,b)$ such that $A \in \mathbb{M}^{\infty}_{i}(\mathcal{K}) \setminus \mathbb{M}^{\infty}_{i-1}(\mathcal{K})$ and $b \in \mathbb{M}^{\infty}_{i-1}(\mathcal{K})$. If two or more elements $b_1,\ldots,b_n$ can be matched with $A$ we always choose $b_j$ having maximal value of filtration. Morse matching $(A,b)$ indicates an interval generated by $(A,b)$, since $b$ generates nontrivial homology class in the level of $g(b)$, which becomes trivial or identical to some other class born earlier in the level of $g(A)$. 

When all the possible matchings are made, in the last for loop the unmatched elements are found. They generate infinite persistence intervals.

It is clear that the levels of filtration need to be processed in order. See the Appendix for more details.

Let us now show that the presented technique can be interpreted in therms of the standard algebraic Algorithm~\ref{alg:matRed}.
\begin{thm}
\label{th:persViaMorse}
Let $\mathcal{K}$ be a filtered chain complex. Let $\mathbb{M}^{\infty}(\mathcal{K})$ be the \emph{iterated} Morse complex described in Section~\ref{sec:IteraterMorseComplex}. Then the persistence intervals of $\mathcal{K}$ and the intervals obtained from $\mathbb{M}^{\infty}(\mathcal{K})$ by the described algorithm are the same.
\end{thm}
\begin{proof}
From Theorem~\ref{thm:optimalMorse} it is clear that the persistence intervals of $\mathcal{K}$ and $\mathbb{M}^{\infty}(\mathcal{K})$ are the same. Let $(A,b)$ be the first Morse matching made by the presented algorithm (i.e. there does not exist a possible matching $(A',b')$ such that $g(A') < g(A)$ and there does not exist $b'$, a face of $A$ with $g(b') > g(b)$).

Let $\mathbb{M}_{(A,b)}(\mathbb{M}^{\infty}(\mathcal{K}))$ be the Morse complex obtained from $\mathbb{M}^{\infty}(\mathcal{K})$ by making a Morse matching $(A,b)$. 
To prove the theorem it suffices to show that the multiset of persistence intervals of $\mathbb{M}^{\infty}(\mathcal{K})$ minus the interval $[g(A),g(b)]$ is equal to the multiset of persistence intervals of $\mathbb{M}_{(A,b)}(\mathbb{M}^{\infty}(\mathcal{K}))$\footnote{This  follows from the fact that $\mathbb{M}_{(A_i,b_i),...,(A_n,b_n)} = \mathbb{M}_{(A_i,b_i)} \circ  ... \circ \mathbb{M}_{(A_n,b_n)}$ . Easy proof is left for the reader.} 

First let us remind how the procedure to compute Morse boundary works and how it is interpreted in matrix-reduction algorithm. The details can be found in~\cite{forman}. Let us assume just one Morse matching, $(A,b)$, is made. And also that $b$ is in the boundary of $A,A_1,\ldots,A_n$. Then boundaries of $A_1,\ldots,A_n$ need to be changed in the following way: $\partial A_i = \partial A_i \setminus b \cup \partial A \setminus b$ i.e. $b$ is replaced in boundary of $A_i$ with boundary of $A$ excluding $b$, as in Figure~\ref{fig:morseBoundary}.

\begin{figure}[!h]
\centerline{\includegraphics[scale=0.5]{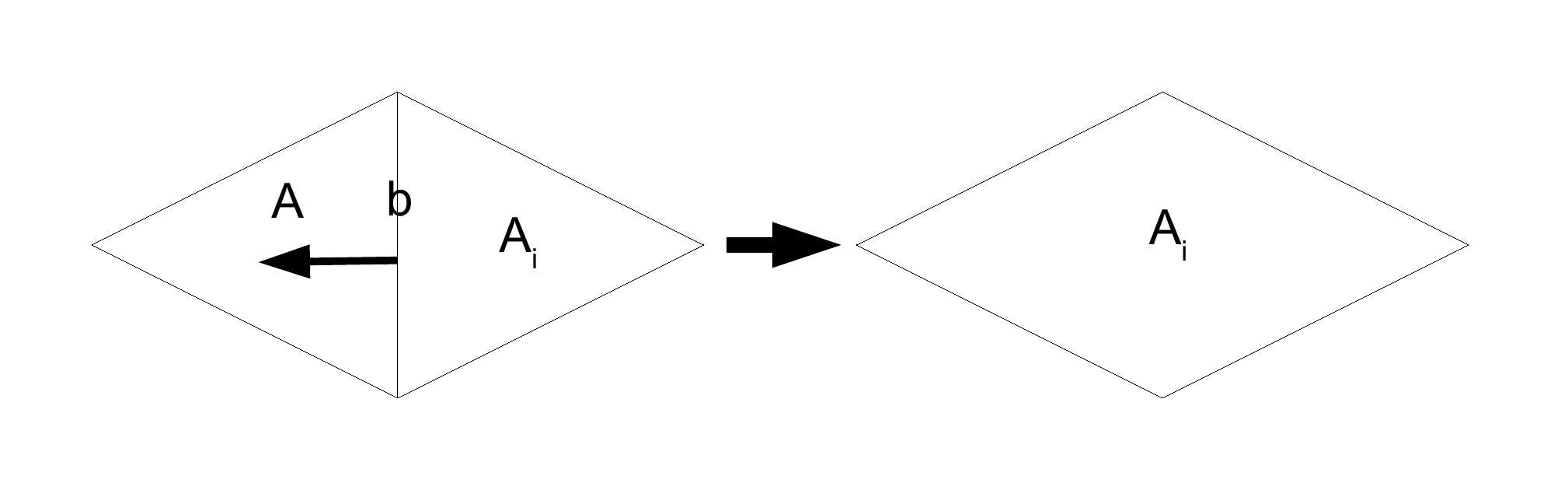}}
\caption{Illustration how performing a single pairing changes the complex.}
\label{fig:morseBoundary}
\end{figure}

Now, suppose Algorithm~\ref{alg:matRed} is run on the complex $\mathbb{M}^{\infty}(\mathcal{K})$. Without loss of generality we may assume that the column $A$ is the first nonzero column in the matrix. Since there cannot be a collision there, Algorithm~\ref{alg:matRed} leaves the column $A$ unchanged and the interval $[g(b),g(A)]$ is obtained in dimension of $b$ -- as in the case of the algorithm presented in this section. But, row $b$ may cause some collisions later in the course of execution of Algorithm~\ref{alg:matRed}. Suppose the first collision in Algorithm~\ref{alg:matRed} occurs:
\[
\begin{array}{cccccc}
 & &  &A&  &A_i\\
 & &  &.&  &.\\
 & &  &.&  &.\\
b&.&..&1&  &1\\
 & &  &0&  &0\\
 & &  &.&  &.\\
 & &  &0&  &0\\
\end{array}
\]
To remove this collision, Algorithm~\ref{alg:matRed} performs the addition $A_i := A_i + A$. In the algorithm presented in this section, when processing cell $A$ the matching $(A,b)$ was made and the Morse boundaries were computed. As one can see, the computations of Morse boundary after the matching $(A,b)$ is simply 	equivalent to summing $A_i := A_i + A$ for all $A_i$ having $b$ in boundary. Therefore all the future collisions caused by $b$ are resolved. Consequently making a Morse matching is equivalent to resolving all the future collisions at once. This simple observation proves the theorem.
\end{proof}

In Figure~\ref{fig:morseToTheEnd} an illustration of the presented procedure on the complex $\mathbb{M}^{\infty}(\mathcal{K})$ from Figure~\ref{fig:iteratedCmplxPers} is given.

\begin{figure}[!h]
\centerline{\includegraphics[scale=0.4]{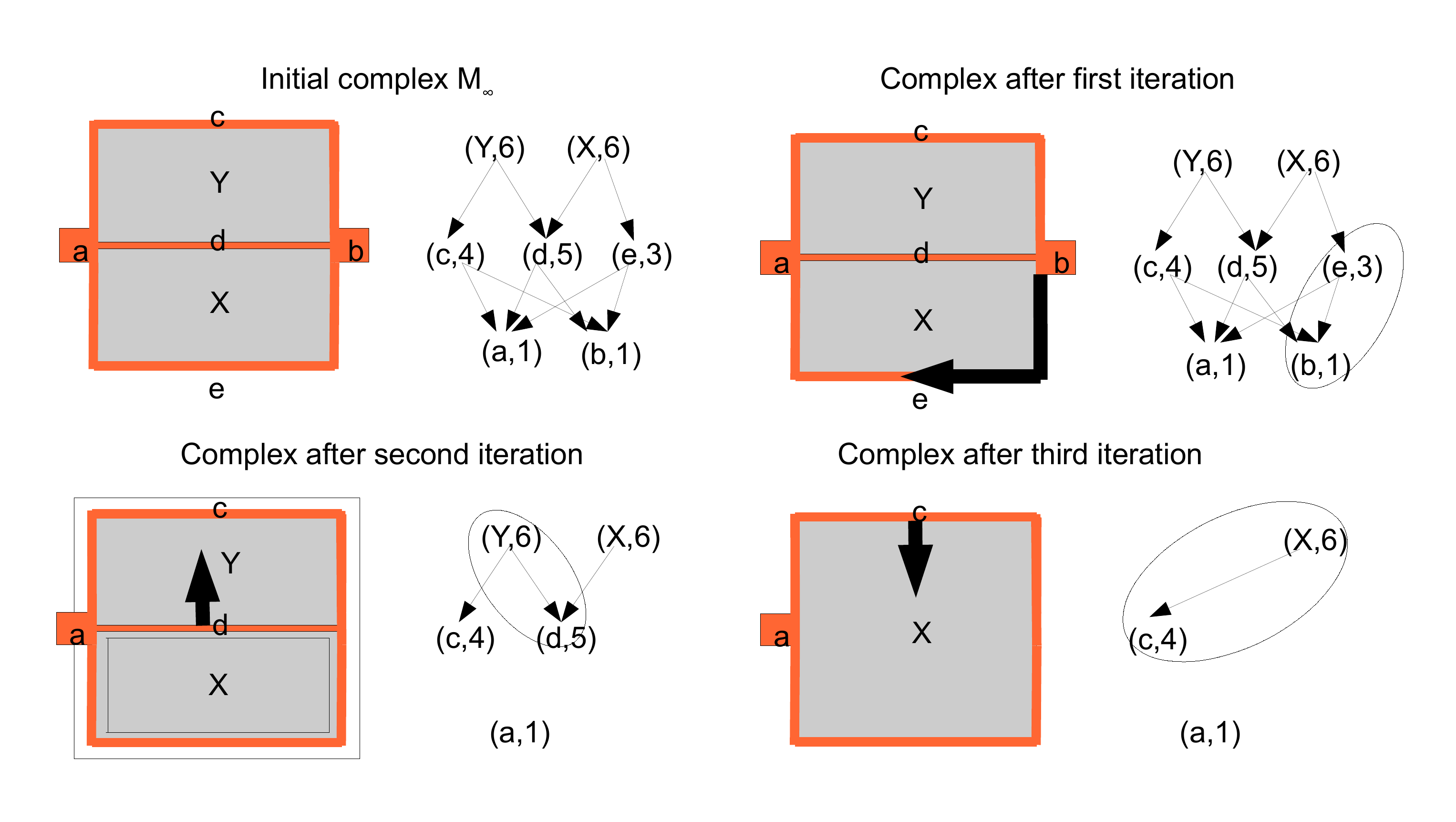}}
\caption{First complex on the top left -- complex $\mathbb{M}^{\infty}(\mathcal{K})$ from Figure~\ref{fig:iteratedCmplxPers}. 
The complex on the top right -- on filtration value $3$ the first possible Morse matching between $b$ and $e$ appears (indicated on the left with arrow, and with ellipse on the right) is made. When the matching is constructed, the persistence interval $[1,3]$ in dimension zero is reported. 
Third complex on the bottom left is obtained as a result. There, on the level $6$ a matching between $Y$ and $d$ can be made (we want to point out that there is also possible matching between $Y$ and $c$, however filtration value of $d$ is higher than filtration value of $c$). After the matching, the persistent interval $[5,6]$ in dimension one is reported.
 The final complex on the bottom right-- the last possible matching between $X$ and $c$ is made. After the matching, the persistent interval $[4,6]$ in dimension $1$ is reported. The remaining cell in the complex is the vertex $a$ it correspond to infinite persistence interval $[1,\infty]$ in dimension zero.}
\label{fig:morseToTheEnd}
\end{figure} 

We want to point out that this approach promises to parallelize well. The details will be presented in a more technical paper~\cite{dist}.


Moreover, the approach will be as scalable as the implementations of the underlying graph algorithms. Therefore, using the available, mature libraries, we hope to achieve good practical performance. 

We also want to point out that Algorithm~\ref{alg:persistenceViaMorse} can be used with minor modification for initial filtered complex $\mathcal{K}$. Easy changes are left for the reader.





\section{Conclusions}
\label{sec:conclusions}

In our opinion the presented technique has several advantages, comparing to the standard way of computing homology and persistence:
\begin{enumerate}
\item It is combinatorial, does not require any algebraic matrix operation. 
\item It is based on graph theory -- we can use efficient algorithm (exact and approximate) and their existing implementations -- in particular libraries for distributed graph operations~\cite{graph_lib}.
\item It is intuitive -- it is easy to visualize the process of homology computations. 
\end{enumerate}

We are aware of some drawbacks and complications of our method:
\begin{enumerate}
\item There may exist bad cases, where the complexity will be unsatisfactory. 
\item This technique might not be suitable for cubical data. Existing methods~\cite{gunther,wagner11} rely on the compact representation of cubical grids. In our case the complex need not be a cubical complex after the first iteration, preventing us from storing it efficiently. 
\end{enumerate}

The presented techniques can be used to formalize and generalize homology-preserving properties of graph pyramids used in image recognition~\cite{peltier}. Moreover they can be beneficial in verified homology computation~\cite{heras} by avoiding matrix operations which are costly to verify automatically.

Summarizing, in this paper a novel technique of homology and persistent homology computations has been presented. It indicates that problem of computing (persistent) homology can be solved by iteratively applying standard graph algorithms. We hope that this approach will  lead to a scalable implementation for (persistent) homology computations.

\section*{Acknowledgments}
The authors would like to thank Herbert Edelsbrunner and Urlich Bauer for their valuable comments and suggestions. Both authors are supported by Google Research Awards program. We thank Marian Mrozek for the supervision of this project. PD. is partially supported by grant IP 2010 046370. HW. is partially supported by Foundation for Polish Science IPP Programme "Geometry and Topology in Physical Models".

\section{Appendix}
\label{sec:appendix}

\subsection{Example of matrix reduction computations}
A simple example of persistence intervals computations with the Algorithm~\ref{alg:matRed} are presented in Figure~\ref{fig:persistenceAlgorithm}. On the left, the initial complex. We assume that the filtration value for every vertex is $0$. The filtration value of edges are given in the picture. On the upper left, the initial boundary matrix. As one can see, the only collision is between columns $cd$ and $bd$. Therefore we have column $cd = cd + bd$ on the upper right. Then a collision between $cd$ and $ac$ appears and we set $cd = cd + ac$ on the lower left. There again we have a collision $cd$ with $ab$ which is removed in the lower right by setting $cd = cd + ab$. On the matrix in lower right there are no more collisions, therefore we can read persistence intervals out of it. Lowest one in column $ac$ indicates that edge $ac$ kills connected component created by $c$, which gives an interval $[0,1]$ in dimension $0$. Analogously lowest one in column $bd$ induces an interval $[0,1]$ in dimension $0$. Lowest one in column $ab$ indicates that edge $ab$ kills a connected component created in $b$, which gives an interval $[0,2]$ in dimension $0$. Zero column $cd$ induces an infinite interval $[3,\infty]$ in dimension one. 
\begin{figure}[!h]
\centerline{\includegraphics[scale=0.5]{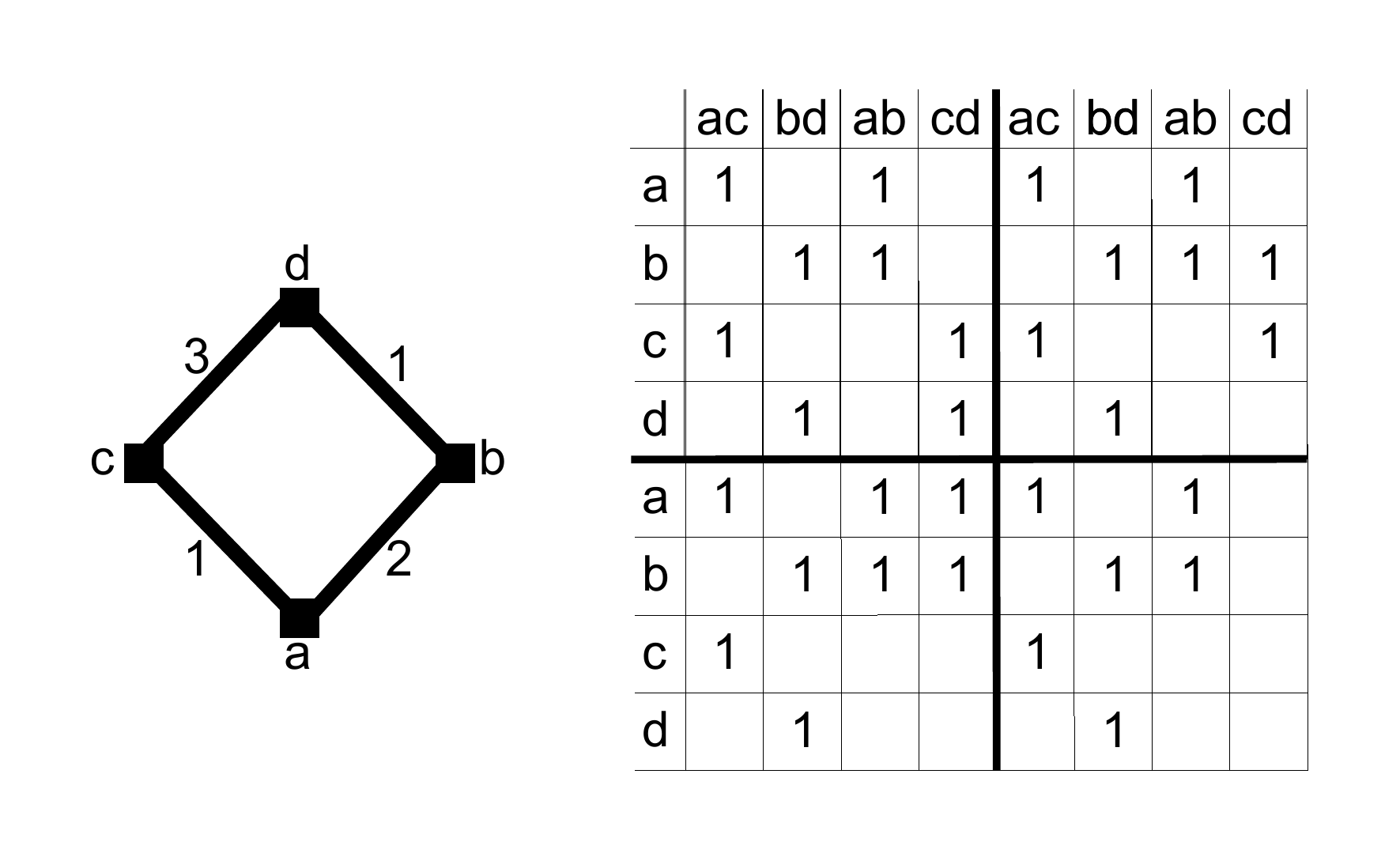}}
\caption{Example of matrix reduction computations}
\label{fig:persistenceAlgorithm}
\end{figure}

\subsection{Relaxing the tolerance}
In Figure~\ref{fig:epsilonTolerancePairing} we show that if one makes a Morse matchings between elements $(A,M(A))$ such that $|g(A)-g(M(A))| \leq \epsilon$, one may get arbitrarily large differences in the output persistence intervals. 
\begin{figure}[!h]
\centerline{\includegraphics[scale=0.5]{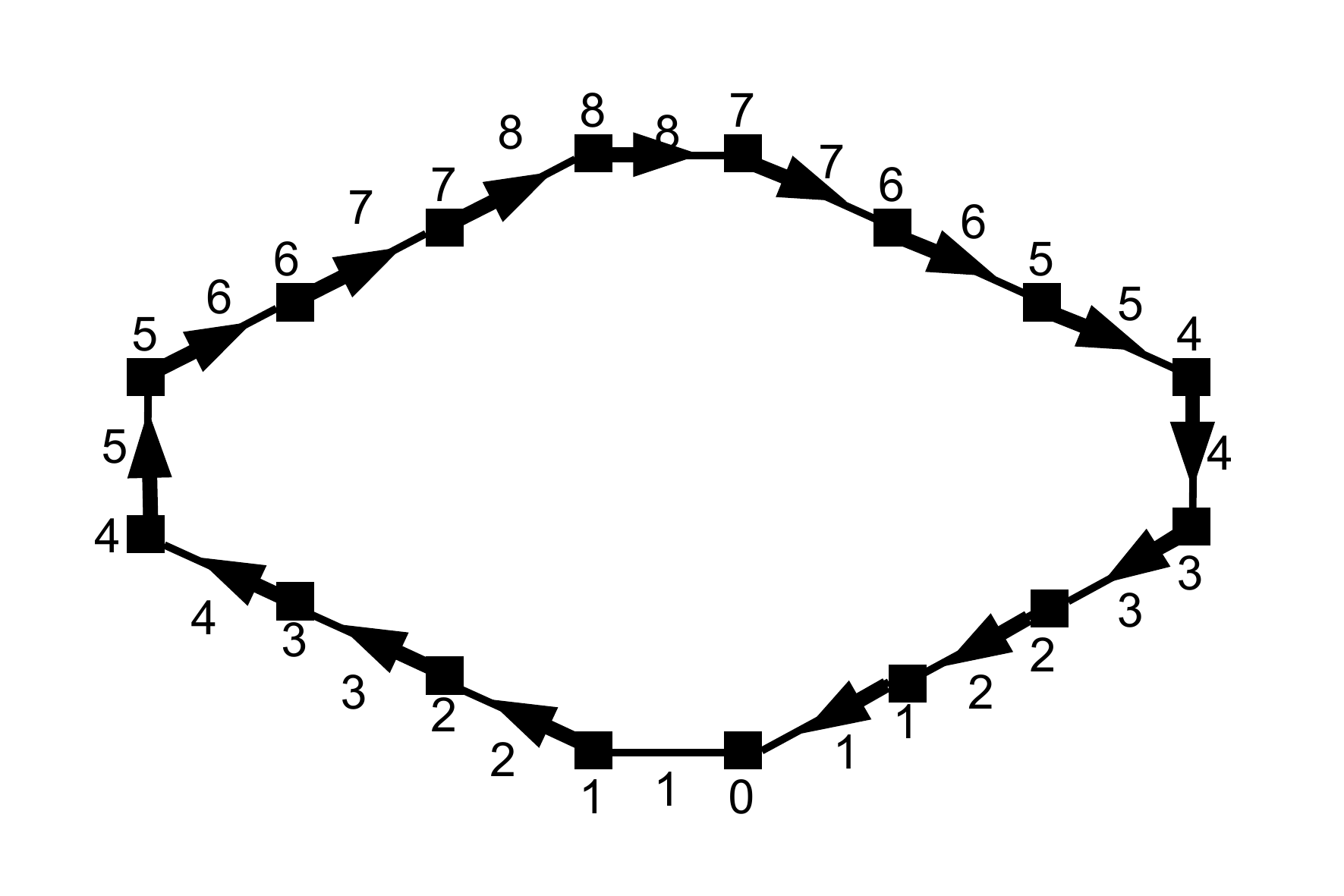}}
\caption{Let us assume the matchings are made with tolerance $\epsilon = 1$ (i.e. one is allowed to do Morse matchings between elements having the absolute value of difference of filtration values not greater than $\epsilon$). In the original complex in dimension $1$ we have a persistence interval $[8,\infty]$, while in the Morse complex (matchings are indicated with arrows) we have a persistence interval $[1,\infty]$. It is clear that by enlarging the circle, one may get arbitrary large difference in persistence intervals.}
\label{fig:epsilonTolerancePairing}
\end{figure}

\subsection{Processing order}
In Figure~\ref{fig:morseForIntervalsOrderNeeded} it is shown what happens if we do not proceed in Algorithm~\ref{alg:persistenceViaMorse} in the filtration order. 
\begin{figure}[!h]
\centerline{\includegraphics[scale=0.5]{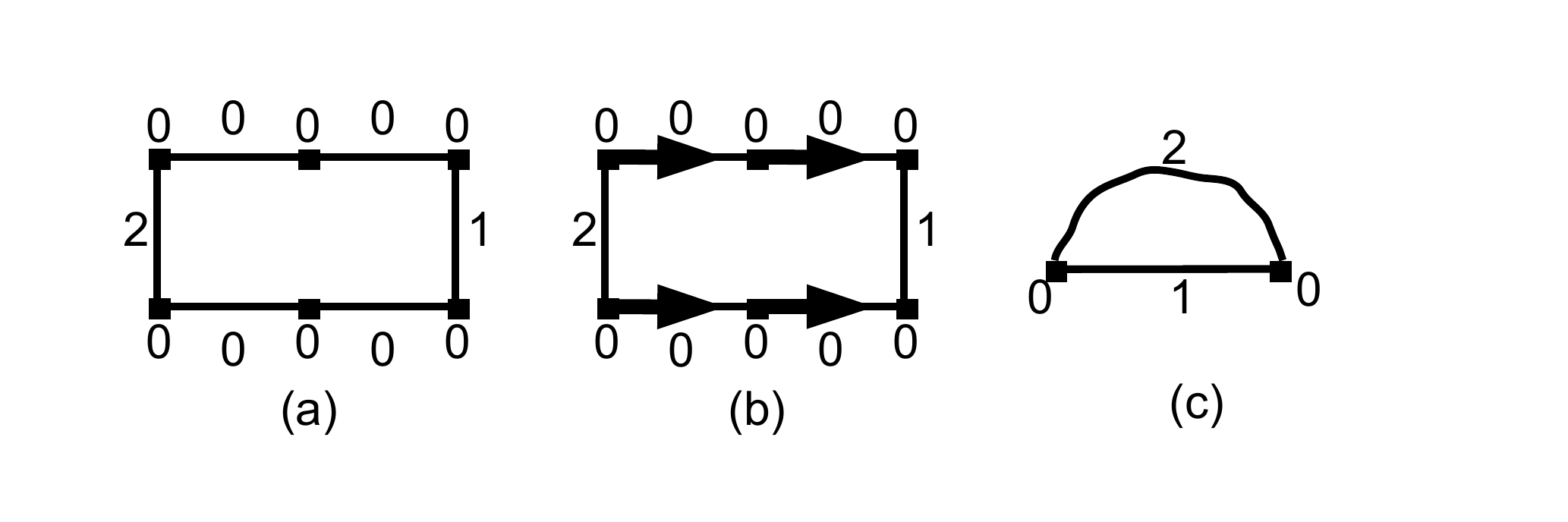}}
\caption{(a) The initial complex $\mathcal{K}$. Numbers indicate filtration values of elements. (b) Directed path on $\mathcal{K}$ compatible with filtration. (c) The reduced complex $\mathbb{M}^{\infty}(\mathcal{K})$. Suppose $\mathbb{M}^{\infty}_2(\mathcal{K}) \subseteq \mathbb{M}^{\infty}(\mathcal{K}) $ is considered first in the described procedure. Then a Morse matching can be made between one dimensional cell having filtration value $2$ and one of vertices of filtration value $0$ (and an interval $[0,2]$ in dimension $0$ is reported). Consequently no more matching can be made at any level of filtration and an interval $[0,\infty]$ in dimension $0$ and $[1,\infty]$ in dimension $1$ are reported. This is wrong, since the true output is $[0,1]$, $[0,\infty]$ in dimension $0$ and $[2,\infty]$ in dimension $1$. Therefore indeed the complexes have to be processed in the order of filtration values.}
\label{fig:morseForIntervalsOrderNeeded}
\end{figure}

\end{document}